\documentclass[conference,a4paper]{IEEEtran}

\PassOptionsToPackage{cmex10}{amsmath}
\usepackage{cascade}
\usepackage[utf8]{inputenc} 
\usepackage[T1]{fontenc}
\usepackage{url}              
\usepackage{cite}             

\usepackage{amssymb,amsbsy,amsthm,dsfont}
\usepackage{mleftright}       
\mleftright                   

\usepackage{graphicx}         
\usepackage{booktabs}         

\begin{document}

\title{Quickest Inference of Susceptible-Infected Cascades in Sparse Networks}

 \author{%
   \IEEEauthorblockN{Anirudh Sridhar\IEEEauthorrefmark{1},
                     Tirza Routtenberg\IEEEauthorrefmark{2},
                     and H. Vincent Poor\IEEEauthorrefmark{1}}
   \IEEEauthorblockA{\IEEEauthorrefmark{1}%
                     Department of Electrical and Computer Engineering,
                     Princeton University,
                     Princeton, NJ,
                     \{anirudhs, poor\}@princeton.edu}
   \IEEEauthorblockA{\IEEEauthorrefmark{2}%
                     Department of Electrical and Computer Engineering,
                     Ben-Gurion University of the Negev,\\
                     Beer-Sheva, Israel,
                     tirzar@bgu.ac.il}
}

\maketitle

\begin{abstract}
  We consider the task of estimating a network cascade as fast as possible.
  The cascade is assumed to spread according to a general Susceptible-Infected process with heterogeneous transmission rates from an unknown source in the network.
  While the propagation is not directly observable, noisy information about its spread can be gathered through multiple rounds of error-prone diagnostic testing. 
  We propose a novel adaptive procedure which quickly outputs an estimate for the cascade source and the full spread under this observation model.
  Remarkably, under mild conditions on the network topology, our procedure is able to estimate the full spread of the cascade in an $n$-vertex network, \emph{before} $\poly \log n$ vertices are affected by the cascade. We complement our theoretical analysis with simulation results illustrating the effectiveness of our methods.
\end{abstract}

\section{Introduction}

Large-scale networks are often vulnerable to \emph{cascading failures}, where anomalous behavior originating from a small set of nodes spreads rapidly to the rest of the network.
Left unchecked, network cascades can have a devastating impact on society.
This has been made painfully clear by the ongoing COVID-19 pandemic, as well as through other examples including the diffusion of misinformation in social networks \cite{conspiracy, Fourney2017,  tacchini2017some} and malware in cyber-physical networks \cite{KW91,  wang_viral_modeling, codered_propagation}. It is therefore of central importance to accurately locate network cascades \emph{before} too many nodes are compromised.

Unfortunately, it is often the case that information about a cascade is noisy in the early stages of its spread, which can make the cascade challenging to locate. 
Consider, for instance, the use of diagnostic testing to determine whether individuals in a population are infected with a contagious disease. As seen from the inaccuracies of early antigen-based rapid diagnostic tests for the detection of COVID-19 as well as other diseases \cite{who, rapid_tests}, such tests may have a significant false positive or negative rate. 
In recent work \cite{sridhar_poor_bayes, sridhar_poor_sequential, sridhar2022qse}, Sridhar and Poor designed sequential estimators for the \emph{source} of a network cascade which takes such uncertainties into account while also quickly coming to a decision so that mitigation measures (e.g., quaranting) can be applied in a timely manner. 
However, a serious weakness of their work is that rigorous guarantees on the performance of their estimators were only established for extremely simple cascade dynamics and network topologies.
In particular, it has remained unclear whether estimators with provable performance guarantees exist in practical settings. 

In the present work, we address this gap by designing a novel sequential estimation procedure for realistic network cascades in generic networks with bounded degree. 
We assume that the cascade is modeled by a heterogeneous Susceptible-Infected (SI) process that can describe multi-type agents and viral mutations \cite{allard2009heterogenous, alexander2010risk, eletreby2020effects, tian2021analysis, yagan2021modeling}.
As the cascade spreads, the behaviors of nodes are periodically monitored through error-prone diagnostic tests. 
From sequential observations of these noisy measurements, our procedure quickly outputs an accurate estimate for the cascade source, as well as the full spread of the cascade. Notably, under minimal assumptions on the cascade dynamics and network topology, we show that the \emph{full spread} of the cascade can be reliably estimated in an $n$-vertex graph before $\poly \log n$ vertices are affected. 
We validate these theoretical results through simulations, demonstrating that our estimator can quickly locate SI cascades in random regular graphs before a significant fraction of nodes is affected.

\subsection{Related work}

Our work contributes to the literature on the \emph{quickest inference of network cascades}, where the overarching goal is to leverage both the network topology and noisy, vertex-level signals to quickly infer aspects of the cascade. 
In \cite{zou2020qd, ZouVeeravalli2018, ZVLT2018}, Zou, Veeravalli, Li and Towsley derived near-optimal algorithms for detecting the emergence of a network cascade when the cascade spreads slowly through the network. 
See also \cite{Rovatsos2020quickest, rovatsos2021quickest, xian2022adaptive, rovatsos2021varying, zhang2022detection} for extensions of their initial work.
The most relevant work to ours is that of Sridhar and Poor \cite{sridhar_poor_bayes, sridhar_poor_sequential, sridhar2022qse}, who developed procedures to quickly estimate the source of a simple, deterministic cascade in networks that can be represented as lattices and regular trees. The present work expands upon the methods of \cite{sridhar2022qse} to design estimators for realistic cascades in generic networks with bounded degree.

We also mention a few fascinating directions on the inference of network cascades in different contexts.
In the seminal work of Shah and Zaman \cite{shah2010detecting, shah2011rumors}, the cascade source is estimated from a large, noiseless snapshot of the affected vertices in a tree (in contrast, the present work estimates the source in generic networks \emph{before} too many vertices are affected). Several authors have extended the initial ideas of \cite{shah2010detecting, shah2011rumors} in various fruitful directions; see, e.g., \cite{KhimLoh15, fanti2015spy, wang2014, ying_book, zhu_ying_1}. 
Another line of work approaches cascade detection and estimation by placing a small number of noiseless ``sensors" from which to obtain high-quality measurements of the cascade's impact \cite{CF2010, leskovec2007outbreak, adhikari2019monitoring, haldar2023temporal, heavey2022detection, batlle2022adaptive}. We study a complementary setting where many nodes are monitored but measurements can be quite noisy.

\subsection{Paper organization}

Section \ref{sec:notation} contains notational conventions we use throughout the paper.
Section \ref{sec:models_and_methods} details the cascade and diagnostic testing models, as well as adaptive estimators for the cascade source and full spread.
We study the performance of our estimator both theoretically and empirically in Section \ref{sec:performance_analysis},
and we conclude in Section \ref{sec:conclusion}.
The proofs of our main results can be found in the appendices.

\section{Notation}
\label{sec:notation}

We let $\mathbb{R}, \mathbb{R}_+, \mathbb{Z}, \mathbb{Z}_{\ge 0}$ denote the reals, the positive reals, the integers, and the non-negative integers, respectively.
For a finite set $S$, $|S|$ denotes the number of elements in the set.
Throughout the paper, we use standard asymptotic notation (e.g., $O(\cdot), o(\cdot)$).
Also, for two sequences $\{a_n \}_{n \ge 1}$ and $\{b_n \}_{n \ge 1}$, we say $a_n \lesssim b_n$ if $a_n \le (1 + o(1)) b_n$, where $o(1) \to 0$ as $n \to \infty$.

We represent a graph by a pair $G = (V(G), E(G))$, where $V(G)$ is the vertex set of the graph and $E(G)$ is the edge set. A graph may be either finite or infinite, which means that the vertex set may be finite or infinite in size.
The degree of a vertex $v$, denoted by $\deg_G(v)$, is the number of neighbors $v$ has in $G$. The maximum degree in the graph is denoted by $\Delta$.
For any $u,v \in V(G)$, we write $u \sim v$ to mean that $(u,v)$ is an edge in $G$.
The quantity $\dist(u,v)$ denotes the shortest path distance between $u$ and $v$ in $G$.
The quantity $\cN_v(t)$ is the $t$-hop neighborhood of $v$ in $G$, which consists of all vertices $u \in V(G)$ such that $\dist(u,v) \le t$. Notice that this definition allows for $t$ to take on real values in addition to integer values.
For a finite set $S \subseteq V(G)$, we define $\diam(S) : = \max_{u,v \in S} \dist(u,v)$.

\section{Models and Methods}
\label{sec:models_and_methods}

\subsection{Cascade model}
\label{subsec:cascade}

We model the cascade using a general version of the well-known SI process in networks. In this model, 
the cascade spreads in a stochastic manner and in continuous time. Initially, the cascade consists of just a single vertex $v^*$ (the source). 
Vertices that are adjacent in $G$ interact at random times, and the cascade spreads to a vertex $v$ when $v$ interacts with one of its affected neighbors.

In a bit more detail, the cascade dynamics can be described as follows. Let $\{ \lambda_{uv} \}_{u \sim v}$ be the collection of \emph{interaction rates}, where $\lambda_{uv} \in \mathbb{R}_+$ represents the frequency of interactions between vertices $u$ and $v$. We assume that interactions are symmetric, so that $\lambda_{uv} = \lambda_{vu}$. 
For any time index $t \ge 0$, let us also denote $\cascade(t) \subset V$ to be the set of vertices affected by the cascade at time $t$. When $t = 0$, $\cascade(0) = v^*$; in words, only the source is affected initially. For any $t \ge 0$ and $\delta > 0$ sufficiently small, we have for any vertex $v \in V \setminus \cascade(t)$ that
\begin{equation}
\label{eq:cascade_dynamics}
\p ( v \in \cascade(t + \delta) \vert \cascade(t) ) = \delta \sum_{u \in \cascade(t) : u \sim v} \lambda_{uv} + o ( \delta),
\end{equation}
where $o(\delta) \to 0$ at a faster rate than $\delta \to 0$.
In words, an unaffected vertex $v \in V \setminus \cascade(t)$ becomes affected by the cascade at a rate equal to the sum of the interaction rates with affected neighbors. 

In the special case where all interaction rates are equal ($\lambda_{uv} = \lambda$ when $u \sim v$), the SI model \eqref{eq:cascade_dynamics} has received significant attention in the last century. 
In the mathematical physics community, it is known to be equivalent to
first passage percolation with $\mathrm{Exp}(\lambda)$ edge weights \cite{fpp, fpp_si_equivalence}, and is also equivalent to the (Markovian)  contact process with no recovery \cite{contact_process, markovian_contact_processes}. The process \eqref{eq:cascade_dynamics} has also been used to derive well-known population-level models of viral spread (see, e.g., \cite[Chapter 9]{brauer2012mathematical}).
The usage of \emph{heterogeneous} interaction rates in \eqref{eq:cascade_dynamics} allows us to capture a variety of important scenarios beyond the basic models described above. For instance, the rate at which one individual may infect another can depend significantly on underlying health conditions (e.g, immuno-compromised individuals could be more easily infected), mask-wearing tendencies, and the type of viral strain. In particular, the dynamics \eqref{eq:cascade_dynamics} are closely related to recent work on mask-wearing and viral mutations in network cascades \cite{allard2009heterogenous, alexander2010risk, eletreby2020effects, tian2021analysis, yagan2021modeling}.\footnote{These works examine cases where the \emph{probability} of transmission between two neighbors can be heterogeneous. Since there is a one-to-one correspondence between the probability of transmission and the \emph{rate} of transmission between individuals (see, e.g., \cite{newman}), such models are closely related to \eqref{eq:cascade_dynamics} when the recovery rate is zero.}

A useful property of the cascade is that, after a sufficient amount of time passes, $\cascade(t)$ can be contained within two neighborhoods of the source. 
Before stating this property, let us define $\lambda_{min} : = \min_{u \sim v} \lambda_{uv}$ and $\lambda_{max} : = \max_{u \sim v} \lambda_{uv}$.
The proof can be found in Appendix \ref{sec:cascade_behavior}.

\begin{proposition}
\label{prop:Ek}
Let $\Delta$ be the maximum degree in $G$. Set
\begin{equation}
\label{eq:alpha_beta}
(\alpha, \beta) : = \left( \frac{\lambda_{min} }{12 \log \Delta}, 3 \Delta \lambda_{max} \right),
\end{equation}
and define the event 
\[
\cE_k : = \left \{ \forall t \ge k, \cN_{v^*}(\alpha t) \subseteq \cascade(t) \subseteq \cN_{v^*}(\beta t) \right \}.
\]
Then $\lim_{k \to \infty} \p( \cE_k ) = 1$.
\end{proposition}

\subsection{Observation model}

For each integer value of $t$ and for each $v \in V$, we assume an error-prone diagnostic test is administered to $v$ with probability $p$. The diagnostic test is correct (that is, it outputs a value of $1$ if $v \in \cascade(t)$ and $-1$ if $v \notin \cascade(t)$) with probability $1 - \epsilon$. With probability $\epsilon$, the diagnostic test is incorrect (both false alarm and misdetection errors). Formally, we denote the measurement corresponding to $v$ at time $t$ by $Y_v(t)$, where $Y_v(t) = 0$ if no test is administered to $v$ and $Y_v(t)$ is equal to the output of the diagnostic test (either $-1$ or $+1$) otherwise. In particular, if $v \notin \cascade(t)$, we have that
\begin{align}
\label{Q_minus}
\p( Y_v(t) = 0 \vert \cascade(t) ) & = 1 - p,\nonumber \\
 \p( Y_v(t) = -1 \vert \cascade(t) ) & = p (1 - \epsilon) ,\nonumber \\
 \p( Y_v (t) = +1 \vert \cascade(t) ) & = p \epsilon.
\end{align}
As a shorthand, we say that $Y_v(t) \sim Q^-$, where $Q^-$ is the probability mass function (PMF) described in \eqref{Q_minus}. Similarly, if $v \in \cascade(t)$, we have that
\begin{align}
\label{Q_plus}
\p( Y_v(t) = 0 \vert \cascade(t) ) & = 1 - p, \nonumber\\
 \p( Y_v(t) = -1 \vert \cascade(t) ) & = p \epsilon ,\nonumber \\
 \p( Y_v (t) = +1 \vert \cascade(t) ) & = p(1- \epsilon).
\end{align}
In this case, we say that $Y_v(t) \sim Q^+$, where $Q^+$ is the PMF described in \eqref{Q_plus}. 
For brevity, we will also denote $\mathbf{Y}(t) : = \{ Y_w(t) \}_{w \in V(G)}$ to be the collection of all signals collected at a positive integer time $t$.

\subsection{Source estimation}

Any estimation procedure can be represented by a pair $(T, \widehat{\mathbf{v}})$. Here, $T \in \Z_{\ge 0}$ is an integer-valued, data-dependent stopping time and $\widehat{\mathbf{v}} = \{ \widehat{v}(t) \}_{t \ge 0}$ is a sequence of estimators for the cascade source, where $\widehat{v}(t) \in V(G)$ is a measurable function of the data, i.e., of the signals observed until time $t$. The procedure $(T, \widehat{\mathbf{v}})$ collects signals until the stopping time $T$ is reached, at which point $\widehat{{v}}(T)$ is outputted as an estimate for the cascade source. 
We will also assume that there is a finite \emph{candidate set} $U \subseteq V(G)$, which represents a known set of potential source vertices (i.e., it is known that $v^* \in U$). When $G$ is a finite graph, we may set $U = V(G)$, though in our theoretical results, we will consider infinite graphs as well. 
Our goal is to find an estimator $\widehat{v}(T)$ for the cascade source that enjoys low estimation error, as measured by the graph distance to the source, $\dist(v^*, \widehat{v}(T))$, while also ensuring that a decision is reached as fast as possible (i.e., $T$ is not too large) to prevent the cascade from affecting too many vertices. 
As we shall see in Section \ref{subsec:full_cascade_estimation}, source estimation can also be used as a subroutine to accurately estimate the \emph{full spread} of the cascade.

Our source estimator is based on two guiding principles:
\begin{enumerate}
    \item Vertices close to the source should have many positive cases in their local neighborhood.
    \item If the number of positive cases in a neighborhood of $v$ is significantly greater than the number of positive cases in a neighborhood of $u$, then $v$ is more likely to be close in proximity to the source.
\end{enumerate}

Following the first guiding principle, we start by constructing a \emph{score function} for each vertex based on the positive and negative cases in a local neighborhood. For a vertex $v \in U$ and any positive integer time index $t$, define
\[
Z_v(t) : = \sum_{s = 0}^t \sum_{w \in \cN_v(\alpha s)} Y_w(s).
\]
In words, $Z_v(t)$ is the cumulative sum of \emph{net positive cases} in the local neighborhood $\cN_v(\alpha s)$, for $0 \le s \le t$. The parameter $\alpha$ is chosen as in \eqref{eq:alpha_beta} so that, if $v^* = v$, the vertices in $\cN_v(\alpha s)$ are likely all affected by the cascade for $s$ sufficiently large in light of Proposition \ref{prop:Ek}. 
On the other hand, if $\dist(v^*, v) \ge (\alpha + \beta )t$ (with $\beta$ defined in \eqref{eq:alpha_beta}), then Proposition \ref{prop:Ek} implies that $\cN_v(\alpha t)$ likely has no overlap with $\cC(t)$. It follows that the differences between score functions, measured by $Z_v(t) - Z_u(t)$, is most informative in estimating the source when $\dist(u,v) \ge (\alpha + \beta) t$.

Following the second guiding principle, our estimation procedure will collect data until the score of some vertex $v \in U$ is significantly larger than the score of all vertices that are sufficiently far from $v$.
Specifically, for each $v \in U$ we define the stopping time $T^{\alpha \beta}(v)$, which halts at the first time $t \ge 0$ when the following condition is satisfied:
\begin{equation}
\label{eq:Tv_stopping_time}
Z_v(t) - Z_u(t) \ge \frac{2 \log |U| }{\log \left( \frac{1 - \epsilon}{\epsilon} \right)}, \hspace{0.2cm} \forall u \in U : \dist(u,v) \ge (\alpha + \beta)t.
\end{equation}
Our estimation procedure is formally given by $(T^{\alpha \beta}, \widehat{\mathbf{v}} )$, where $T^{\alpha \beta} : = \min_{v \in U} T^{\alpha \beta}(v)$ and $\widehat{v}(T^{\alpha \beta}) \in \argmin_{v \in U} T^{\alpha \beta}(v)$. 

\begin{remark}
An edge case in the description of $T^{\alpha \beta}(v)$ is that when $t > \dist(u,v) / (\alpha + \beta)$, the condition in \eqref{eq:Tv_stopping_time} is vacuous. Consequently, if none of the stopping times halt for any $t \le \diam(U) / (\alpha + \beta)$, we set $T^{\alpha \beta} = \diam(U) / (\alpha + \beta)$ and $\widehat{v}(T^{\alpha \beta})$ is an arbitrary vertex from $U$. However, as our results show, it is quite unlikely that such an edge case would occur (see Theorem \ref{thm:main_performance_analysis}).
\end{remark}

\begin{remark}
\label{remark:distance_bound}
The threshold for $Z_v(t) - Z_u(t)$ in \eqref{eq:Tv_stopping_time} is carefully chosen so that $\dist(v^*, \widehat{v}(T^{\alpha \beta})) \le (\alpha + \beta)T^{\alpha \beta}$ with probability tending to 1 as $|U| \to \infty$.
For details, see Lemma \ref{lemma:estimation_tail_bound}.
\end{remark}

\begin{remark}
In \cite{sridhar_poor_sequential, sridhar2022qse} it was noted that source estimation can be viewed as a sequential multi-hypothesis testing problem, for which a procedure based on the computation of log-likelihood ratios of the observations achieves optimal performance. Indeed, in the simple case where $\cascade(t) = \cN_{v^*}(t)$ and $\alpha = 1$, $Z_v(t) - Z_u(t)$ is proportional to the log-likielihood ratio between the measures $\p ( \cdot \vert v^* = v)$ and $\p ( \cdot \vert v^* = u)$. However, computing log-likelihood ratios for the more realistic, stochastic cascade models considered in this work would require us to integrate over all possible realizations of the cascade. We instead use $Z_v(t) - Z_u(t)$ as a proxy for the log-likelihood ratio, which circumvents these issues, while also achieving similar performance guarantees to the likelihood-based procedure of \cite{sridhar_poor_sequential, sridhar2022qse}.
\end{remark}

\subsection{From source estimation to full cascade estimation}
\label{subsec:full_cascade_estimation}

A natural and important goal is also to quickly estimate the \emph{full spread} of the cascade, rather than just the source. This goal can be readily achieved using the procedure $(T^{\alpha \beta}, \widehat{\mathbf{v}})$ as a subroutine. Indeed, recall from Proposition \ref{prop:Ek} that if $T^{\alpha \beta}$ is sufficiently large, then $\cascade(T^{\alpha \beta}) \subseteq \cN_{v^*}(\beta T^{\alpha \beta})$ with high probability. If $\dist(v^*, \widehat{v}(T^{\alpha \beta})) \le (\alpha + \beta) T^{\alpha \beta}$ as suggested by Remark \ref{remark:distance_bound}, it follows that 
\begin{equation}
\label{eq:cascade_estimator}
\cascade(T^{\alpha \beta}) \subseteq \cN_{\widehat{v}(T^{\alpha \beta})} ((\alpha + 2 \beta) T^{\alpha \beta} ) =: \widehat{C}(T^{\alpha \beta}),
\end{equation}
with high probability. Hence the estimator $\widehat{\cascade}(T^{\alpha \beta})$ for the cascade spread \emph{fully contains} the true spread in this case. This intuition is confirmed in Theorem \ref{thm:main_performance_analysis} below.

\section{Performance analysis}
\label{sec:performance_analysis}

\subsection{Theoretical results}

Although our algorithm applies to \emph{any} graph with bounded degree, in our theoretical results we focus on a general class of \emph{infinite} graphs. Such graphs capture scenarios where the cascade is small relative to the full network size without overly complicating the mathematical analysis,\footnote{In finite graphs, the cascade will spread to all the vertices in finite time, at which point no new information about the source location can be learned from the data. The study of infinite graphs allows us to avoid such ``boundary effects".}
though we expect that our results also hold for finite graphs with sufficiently large diameter.
Formally, we assume the following:

\begin{assumption}
Assume that $G$ has infinitely many vertices, is connected, and has a finite maximum degree $\Delta$.
\end{assumption}

We next define a few key quantities. For a non-negative integer $t$ and a vertex $v \in V(G)$, define $f_v^\alpha (t) : = \sum_{s = 0}^t | \cN_{v}(\alpha s) |$ as well as its inverse function $F_v^{\alpha}$.
It turns out that $f_{v^*}^\alpha$ and $F_{v^*}^\alpha$ play a fundamental role in the performance guarantees of our estimator. 
Indeed, notice from the definition of the score functions that, conditionally on the cascade evolution $\boldsymbol{\cascade} : = \{ \cascade(s) \}_{s \ge 0}$, 
\begin{multline*}
\frac{\E [ Z_{v^*}(t) \vert \boldsymbol{\cascade} ]}{p(1 - 2 \epsilon)}  =  \sum_{s = 0}^t \left( | \cN_{v^*}(\alpha s) \cap \cascade(s) | -  | \cN_{v_*} (\alpha s) \setminus \cascade(s) | \right) \\
= \sum_{s = 0}^t | \cN_{v^*}(\alpha s) | - 2 \sum_{s = 0}^t | \cN_{v_*}(\alpha s) \setminus \cascade(s) | \sim f_{v^*}^\alpha(t),
\end{multline*}
where the final asymptotic expression holds as $t \to \infty$, with high probability in light of Proposition \ref{prop:Ek}. 
On the other hand, if $u \in V(G)$ satisfies $\dist(v^*, u) > (\alpha + \beta)t$, then the vertices in $\cN_u(\alpha t)$ are unaffected by the cascade at time $t$ with high probability (see Proposition \ref{prop:Ek}), hence $\E [ Z_u(t) \vert \boldsymbol{\cascade} ] \le 0$. 
It follows that $\E [ Z_{v^*}(t) - Z_u(t) \vert \boldsymbol{\cascade}] \ge p(1 - 2 \epsilon) f_{v^*}^\alpha (t)$.
If $Z_v(t) - Z_u(t)$ concentrates around its expectation, we see that it exceeds the threshold in \eqref{eq:Tv_stopping_time} when $f_{v^*}^\alpha(t) \gtrsim \log |U|$, or equivalently when $t \lesssim F_{v^*}^\alpha( \log |U|)$. This intuition is confirmed in our first main result below.

\begin{theorem}
\label{thm:main_performance_analysis}
Let $(\alpha, \beta)$ be set according to \eqref{eq:alpha_beta}
and let $U \subset V(G)$ be any finite candidate set of potential source vertices. Then, with probability tending to 1 as $|U| \to \infty$, the following hold:
\begin{enumerate}
    \item \label{item:thm_T}
    $
    T^{\alpha \beta} \le  F_{v^*}^\alpha \left( \frac{ 15  \log |U| }{p (1 - 2 \epsilon)^2} \right)
    $
\\
    \item \label{item:thm_dist}
    $
    \dist(v^*, \widehat{v}(T^{\alpha \beta}) )  \le (\alpha + \beta) F_{v^*}^\alpha \left( \frac{ 15  \log |U| }{p (1 - 2 \epsilon)^2} \right)
    $
\\
    \item \label{item:thm_C}
    $
    \cascade(T^{\alpha \beta} )  \subseteq \widehat{\cascade}(T^{\alpha \beta} )
    $
\end{enumerate}
\end{theorem}

The intuition behind Item \ref{item:thm_T} can be found in the discussion preceding the theorem statement. Notice in particular that as the testing frequency decreases ($p \to 0$) or as testing errors become large ($\epsilon \to 1/2$), the upper bound for $T^{\alpha \beta}$ increases.
Item \ref{item:thm_dist} essentially follows from the choice of thresholds in \eqref{eq:Tv_stopping_time}. As discussed in Remark \ref{remark:distance_bound}, the thresholds ensure that $\dist(v^*, \widehat{v}(T^{\alpha \beta}) \le (\alpha + \beta) T^{\alpha \beta}$, and the statement in the theorem follows from substituting the bound on $T^{\alpha \beta}$ given in Item \ref{item:thm_T}.
Item \ref{item:thm_C} shows that the cascade estimator (defined in \eqref{eq:cascade_estimator}) contains no false negatives. 
The full proof details can be found in Appendix \ref{sec:proof_of_thm}.

It is challenging to obtain a bound on the size of the estimated set in the most general setting of bounded degree graphs. 
Fortunately, under a very mild structural condition on the topology of $G$, we can show that the size of $\widehat{\cascade}(T^{\alpha \beta})$ is at most $\poly \log |U|$ -- an exponential reduction from the initial $|U|$ potential locations for the cascade. This is formalized in the following corollary; the proof is in Appendix \ref{sec:proof_of_thm}.

\begin{corollary}
\label{cor:cascade_estimator_size}
Assume the same conditions as Theorem \ref{thm:main_performance_analysis}, and 
suppose that there exist constants $q,r \ge 1$ such that
\begin{equation}
\label{eq:structural_constraint}
| \cN_u(t) | \le q | \cN_v(t) |^r, \qquad \forall u,v \in V(G), \forall t \ge 0.
\end{equation}
Then with probability tending to 1 as $|U| \to \infty$, there is a constant $c = c(\alpha, \beta, q, r)$ such that $|\widehat{C}(T^{\alpha \beta})| \le \log^c |U|$.
\end{corollary}

At a high level, the condition \eqref{eq:structural_constraint} states that neighborhood sizes are \emph{polynomially equivalent}, in the sense that the size of one neighborhood can be bounded by a (fixed) polynomial of any other neighborhood. A consequence of this definition is that if one neighborhood grows polynomially in $t$ (which is often the case in spatial networks such as lattices), then \emph{all} neighborhoods must grow polynomially in $t$. Similarly, if one neighborhood grows exponentially in $t$ (which is often the case in tree-like networks), then \emph{all} neighborhoods must grow exponentially. We expect such a condition to be trivially fulfilled in models of natural networks.

Finally, we study the implications of Theorem \ref{thm:main_performance_analysis} in simple networks, for which neighborhood sizes can be explicitly calculated.
Notably, the bounds on $T^{\alpha \beta}$ in the following corollaries match the performance of the \emph{optimal} source estimation algorithms derived for deterministic cascades in \cite{sridhar2022qse}. 
The proofs of the corollaries below follow immediately from Theorem \ref{thm:main_performance_analysis}, the asymptotic characterization of $F_v^\alpha$ in Lemma \ref{lemma:Fv_asymptotics}, as well as the expressions for neighborhood sizes for lattices and regular trees found in \cite[Appendix A]{sridhar2022qse}.

\begin{corollary}
Assume the same conditions as in Theorem \ref{thm:main_performance_analysis}, and furthermore assume that $G$ is an infinite regular tree with degree at least 3. The following statements hold with probability tending to 1 as $|U| \to \infty$:
\begin{enumerate}
    \item $T^{\alpha \beta} \lesssim \frac{1}{\alpha} \log \log |U|$ 
\\
    \item $\dist(v^*, \widehat{v}(T^{\alpha \beta})) \lesssim \frac{\alpha + \beta}{\alpha} \log \log |U|$
\\
    \item $\cascade(T^{\alpha \beta} ) \subseteq \widehat{\cascade}(T^{\alpha \beta} )$ and $| \widehat{\cascade}(T^{\alpha \beta} )| \le ( \log |U| )^{ 3 \beta / \alpha}$.
\end{enumerate}
\end{corollary}

\begin{corollary}
Assume the same conditions as in Theorem \ref{thm:main_performance_analysis} and furthermore assume that $G$ is an infinite $\ell$-dimensional lattice. The following statements hold with probability tending to 1 as $|U| \to \infty$:
\begin{enumerate}

\item $T^{\alpha \beta} = O \left( \left( \log |U| \right )^{1/ (\ell + 1)} \right)$ \\

\item $\dist( v^*, \widehat{v}(T^{\alpha \beta} ) ) = O \left( \left( \log |U| \right)^{1/(\ell + 1)} \right)$ \\

\item $\cascade(T^{\alpha \beta} ) \subseteq \widehat{\cascade}(T^{\alpha \beta} )$ and $| \widehat{\cascade}(T^{\alpha \beta} ) | = O \left( \left( \log |U| \right)^{\frac{\ell}{\ell + 1}} \right)$.

\end{enumerate}
\end{corollary}

\begin{figure}[t]
\centering
\begin{subfigure}{0.4 \textwidth}
\centering
\includegraphics[width= \textwidth]{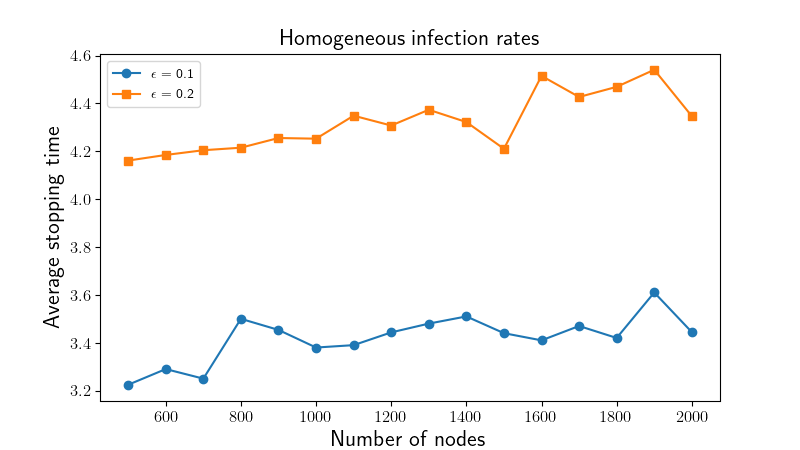}
\caption{}
\label{subfig:homogeneous_time}
\end{subfigure}\\
\begin{subfigure}{0.4 \textwidth}
\centering
\includegraphics[width= \textwidth]{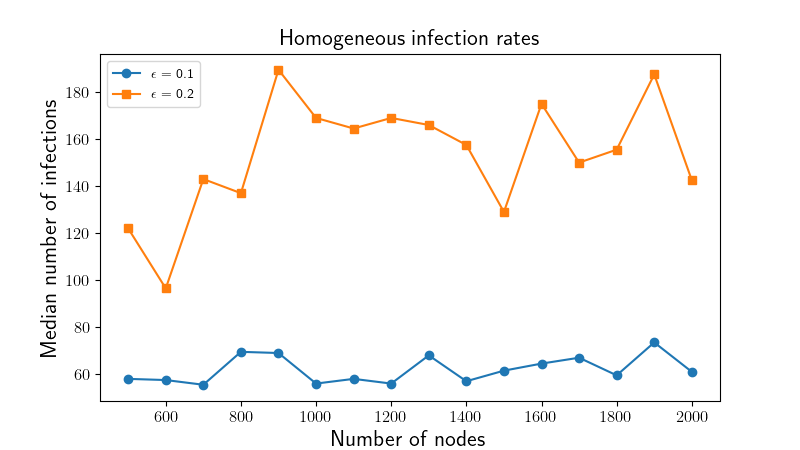}
\caption{}
\label{subfig:homogeneous_infections}
\end{subfigure}
\caption{Plots of the expected stopping time $(a)$ and the median number of infections upon stopping $(b)$ for homogeneous rates. 
}
\label{fig:homogeneous_rates}
\end{figure}

\begin{figure}[t]
\centering
\begin{subfigure}{0.4 \textwidth}
\centering
\includegraphics[width= \textwidth]{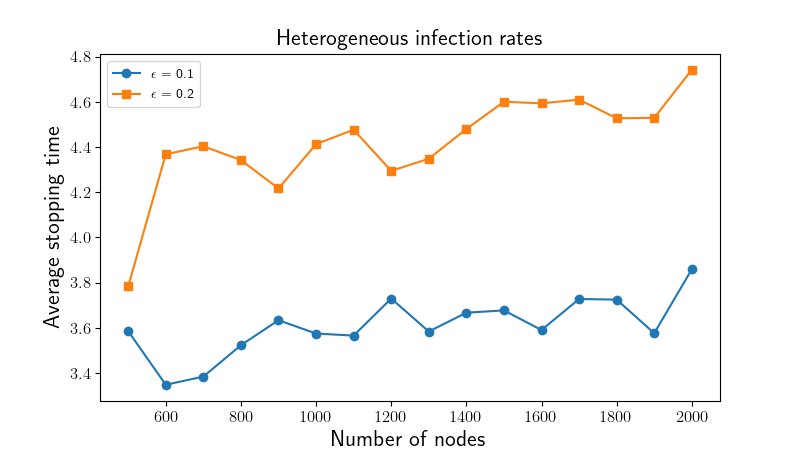}
\caption{}
\label{subfig:heterogeneous_time}
\end{subfigure}\\
\begin{subfigure}{0.4 \textwidth}
\centering
\includegraphics[width= \textwidth]{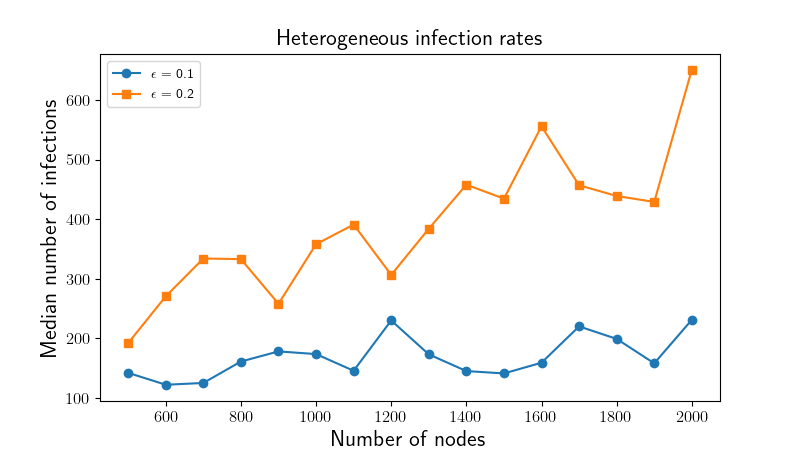}
\caption{}
\label{subfig:heterogeneous_infections}
\end{subfigure}
\caption{Plots of the expected stopping time $(a)$ and the median number of infections upon stopping $(b)$ for heterogeneous rates. 
}
\label{fig:heterogeneous_rates}
\end{figure}

\subsection{Simulations}

To complement our theoretical results, we show through simulations that our cascade estimators have desirable performance in non-asymptotic settings (i.e., small values of $|U|$) as well. 
The underlying network was chosen to be a uniform random 3-regular graph of size ranging between 500 and 2000 nodes.
The cascade source was chosen uniformly at random in each trial. Throughout, we set $\alpha = 1, \beta = 2, p = 0.5$ and studied $\epsilon \in \{0.1, 0.2\}$. 
To generate each data point in our plots, we ran 100 independent trials simulating the cascade propagation and our estimation procedure.\footnote{The curves may appear noisy even after averaging and taking the median due to the large amount of randomness in our simulations, coming from the network structure, cascade evolution, and diagnostic testing model.}

We first studied the performance of our estimator on a (classical) SI model with \emph{homogeneous rates}, where $\lambda_{uv} = 1$ whenever $u \sim v$. 
Our results can be found in Figure \ref{fig:homogeneous_rates}, with the average stopping time found in Figure \ref{subfig:homogeneous_time} and the \emph{median} number of infections in Figure \ref{subfig:homogeneous_infections}. 
In all our simulations, the estimation error was, on average, less than 1.5 for all parameter values.
We plot the median instead of the average to disregard rare instantiations of the cascade that spread extremely rapidly through the network.
Remarkably, when $\epsilon = 0.1$ our estimator is able to reliably track the cascade before 80 vertices are affected, even as the network grows large.
Though the number of infections is substantially larger for $\epsilon = 0.2$, the average stopping time in this case seems to flatten out (see Figure \ref{subfig:homogeneous_time}), indicating that the infection curve should also flatten as the network size increases beyond 2000.

We then studied a setting with \emph{heterogeneous rates}, in which each interaction rate was taken to be a uniform random variable in $[1,1.5]$.
We kept $(\alpha, \beta) = (1,2)$.
Here again, the estimation error was quite small on average, being less than 2 for all tested parameter values.
Interestingly, it does not seem that the heterogeneities in the rates significantly affect the performance of the estimator, as the average stopping times in Figures \ref{subfig:homogeneous_time} and \ref{subfig:heterogeneous_time} are similar.
However, the number of infections in Figure \ref{subfig:heterogeneous_infections} does increase due to the average increase in the spreading rate of the cascade.

\section{Conclusion}
\label{sec:conclusion}
In this work, we considered the problem of estimating a network cascade from a noisy time series of its spread. Prior work on source estimation in this setting only had provable performance guarantees for unrealistically simple cascades and network topologies \cite{sridhar_poor_bayes, sridhar_poor_sequential, sridhar2022qse}. Our work is substantially more general: we develop novel estimators for both the source and the full cascade, with provable guarantees for realistic cascades spreading on arbitrary networks of bounded degree.
Avenues for future work include the study of \emph{optimal} estimators for the cascade source and full spread in the general setting we consider, as well as a development of estimators in scenarios where nodes are adaptively selected, rather than randomly selected, for diagnostic testing.

\clearpage

\bibliographystyle{abbrv}
\bibliography{citations}

\appendices

\section{Proof of Proposition \ref{prop:Ek}}
\label{sec:cascade_behavior}

In this appendix, we characterize the typical behavior of the cascade, formalized through the event $\cE_k$. Formally, for a positive integer $k$, let us define the events 
\begin{align*}
\cE_{1,k} & : = \{ \forall t \ge k, \cN_{v^*}(\alpha t) \subseteq \cascade(t) \} \\
\cE_{2,k} & : = \{ \forall t \ge k, \cascade(t) \subseteq \cN_{v^*}(\beta t) \},
\end{align*}
where $(\alpha , \beta)$ are given by \eqref{eq:alpha_beta}. The main results of this section are the following two lemmas. 
The proofs utilize the alternative representation of the cascade (in terms of exponential edge weights), which we describe in Section \ref{sec:equivalent_cascade}.

\begin{lemma}
\label{lemma:E1k}
For any positive integer $k$, 
\[
\p ( \cE_{1,k}^c) \le \frac{12}{\lambda_{min}} e^{ - k \lambda_{min} / 12}.
\]
\end{lemma}

\begin{lemma}
\label{lemma:E2k}
For any $k$ sufficiently large, it holds that 
\[
\p( \cE_{2,k}^c) \le 900 \left( \frac{2.9}{3} \right)^{3 \Delta \lambda_{max} k}.
\]
\end{lemma}

The proof of Lemmas \ref{lemma:E1k} and \ref{lemma:E2k} can be found in Sections \ref{subsec:E1} and \ref{subsec:E2}, respectively. 
Moreover, the two lemmas readily imply Proposition \ref{prop:Ek}.

\begin{proof}[Proof of Proposition \ref{prop:Ek}]
Notice that $\cE_k = \cE_{1,k} \cap \cE_{2,k}$. Hence, for any $v \in V(G)$, $\p( \cE_k^c) \le \p ( \cE_{1,k}^c ) + \p ( \cE_{2,k}^c) \to 0$ as $k \to \infty$, as desired.
\end{proof}

\subsection{An equivalent cascade model}
\label{sec:equivalent_cascade}

We introduce a useful alternate representation of the cascade dynamics described in Section \ref{subsec:cascade}. We start with some notation. 
Generate a collection of independent random variables $\{X_{uv} \}_{u \sim v }$, where $X_{uv} \sim \mathrm{Exp}( \lambda_{uv})$ denotes the \emph{weight} of the edge $(u,v)$. Given a path\footnote{A path is a finite sequence of distinct vertices $u_1,\ldots, u_k$ such that $u_i \sim u_{i + 1}$ for all $1 \le i \le k - 1$.} $P$ in the graph, we let the weight of the path -- denoted by $\weight(P)$ -- be the sum of the edge weights along the path. Let $\cP_{vu}$ denote the set of all paths starting at $v$ and ending at $u$. Then, in this alternate cascade model, if $v$ is the source, then $u \in \cascade(t)$ if and only if $\inf_{P \in \cP_{vu}} \weight(P) \le t$. In other words, the quantity $\inf_{P \in \cP_{vu}} \weight(P)$ represents the \emph{infection time} of $u$ when $v$ is the cascade source. 
Due to the memoryless property of the exponential distribution, this representation is equivalent to the Markovian dynamics in \eqref{eq:cascade_dynamics}. 
We defer the reader to \cite[Chapter 6]{fpp} and references therein for details on this equivalence.

\subsection{The event $\cE_{1,k}$: Proof of Lemma \ref{lemma:E1k}}
\label{subsec:E1}

Let $u \in \cN_{v^*}(\alpha s)$ and let $d = \dist (u,v)$. Then we can find a path $w_0, \ldots, w_d$ with $w_0 = v^*$ and $w_d = u$.
Notice that the infection time of $u$ is upper bounded by $\sum_{i = 0}^{d - 1} X_{w_i w_{i + 1}}$, where we recall that the $X_{w_i w_{i + 1} }$'s are independent with $X_{w_i, w_{i + 1}} \sim \mathrm{Exp} ( \lambda_{w_i w_{i + 1}} )$. Hence, we have, for any $s \ge 1$ and any $\theta \in (0, \lambda_{min})$, 
\begin{align*}
\p & ( u \notin \cascade(s) )  \le \p_v \left( \sum_{i = 0}^{d - 1} X_{w_i w_{i + 1} } \ge s \right) \\
& \stackrel{(a)}{\le} e^{ - \theta s} \prod_{i = 0}^{d - 1} \frac{ \lambda_{w_i w_{i + 1} } }{ \lambda_{w_i w_{i + 1} } - \theta } \\
& = \exp \left \{ - \theta s - \sum_{i = 0}^{d - 1} \log \left( 1 - \frac{ \theta }{ \lambda_{w_i w_{i + 1} } } \right) \right \} \\
& \stackrel{(b)}{\le} \exp \left \{ - \theta s - \alpha s \log \left( 1 - \frac{\theta}{\lambda_{min} } \right ) \right \}.
\end{align*}
In $(a)$, we have used the form of the moment generating function of an exponential random variable, and in $(b)$ we have used that the function $x \mapsto \log (1 - x)$ is decreasing for $x \in [0,1)$ as well as $d \le \alpha s$. 
Setting $\theta = \lambda_{min} / 2$, it holds for $\alpha \le \lambda_{min} / ( 6\log 2)$, that
\[
\p ( u \notin \cascade(s) ) \le e^{- s \lambda_{min} / 3}.
\]
As a consequence, we also have that
\begin{multline*}
\E \left[ | \cN_{v^*}(\alpha s) \setminus \cascade(s) | \right]  = \sum_{u \in \cN_{v^*}(\alpha s) } \p( u \notin \cascade(s) ) \\
  \le e^{ - s \lambda_{min} / 6} | \cN_{v^*}(\alpha s) | 
 \le 2 \left( e^{ - \lambda_{min} / 6} \Delta^\alpha \right)^s.
\end{multline*}
Above, the final inequality uses the bound $|\cN_{v^*}(\alpha s)| \le 2 \Delta^{\alpha s}$ (see Lemma \ref{lemma:neighborhood_bound}). Markov's inequality now implies
\[
\p \left( | \cN_{v^*}(\alpha s) \setminus \cascade(s) | \ge 1 \right) \le 2 \left( e^{- \lambda_{min}/6} \Delta^\alpha \right)^s \le 2 e^{ - \frac{s \lambda_{min}}{12}},
\]
where the final expression follows from substituting $\alpha = \lambda_{min} / (12 \log \Delta)$. We conclude by taking a union bound over all $s \ge k$.
\qed

\subsection{The event $\cE_{2,k}$: Proof of Lemma \ref{lemma:E2k}}
\label{subsec:E2}

We start by proving a useful intermediate result on lower tail bounds of a sum of independent exponential random variables. 

\begin{lemma}
\label{lemma:weight_lower_bound}
Let $X_1, \ldots, X_m$ be independent random variables with $X_i \sim \mathrm{Exp}(\mu_i)$, and let $\mu : = \max_{1 \le i \le m} \mu_i$. Then for any $\epsilon \in (0, \mu^{-1} )$, it holds for $m$ sufficiently large that
\[
\p \left( \sum_{i = 1}^m X_i \le \epsilon m  \right) \le \sqrt{m} (2.8 \epsilon \mu )^m. 
\]
\end{lemma}

\begin{proof}
Notice that if $\mu_1 \le \mu_2$, where $X_1 \sim \mathrm{Exp}( \mu_1)$ and $X_2 \sim \mathrm{Exp} ( \mu_2)$, then $X_2 \preceq X_1$. Due to this stochastic ordering, $\sum_{i = 1}^m X_i$ stochastically dominates a sum of i.i.d. $\mathrm{Exp}(\mu)$ random variables, which is equal in distribution to $W \sim \mathrm{Gamma}(m, \mu)$. Hence
\begin{align*}
\p  \left( \sum_{i = 1}^m X_i \le \epsilon m \right)  & \le \p ( W \le \epsilon m )   = \int_0^{\epsilon m } \frac{ \mu^m x^{m-1} e^{- \mu x} }{ (m-1)! }dx.
\end{align*} 
Notice that the density of $W$ is increasing for $0 \le x \le (m - 1) / \mu$. As a result, for $\epsilon < 1/\mu$ and $m$ sufficiently large, we can bound the integral by 
\[
\frac{ \mu^{m} (\epsilon m)^{m} e^{ - \mu \epsilon m} }{( m - 1)!} \sim \frac{ \sqrt{m - 1} }{e \sqrt{2 \pi} } (  \epsilon \mu e^{1 - \mu \epsilon } )^m \le \sqrt{m} ( 2.8 \epsilon \mu )^m,
\]
where the asymptotic expansion is due to Stirling's formula. 
\end{proof}

We are now ready to prove our main result. 

\begin{proof}[Proof of Lemma \ref{lemma:E2k}]
Let $\cP_d$ be the set of paths of length $d$ starting from $v^8$. Notice that if, for all $P \in \bigcup_{d \ge \beta t} \cP_d$ we have that $\mathsf{weight}(P) > t$, then $\cascade(t) \setminus \cN_{v^*}(\beta t) = \emptyset$. Our proof therefore bounds the probability that a path $P \in \bigcup_{d \ge \beta t} \cP_d$ has weight at most $t$. To this end, we have that 
\begin{align}
\p &  ( \cascade(t) \setminus \cN_{v^*}(\beta t) \neq \emptyset ) \nonumber \\
& \le \p  \left( \exists P \in \bigcup_{d \ge \beta t} \cP_d : \mathsf{weight}(P) \le t \right) \nonumber \\
& \le \sum_{d \ge \beta t} | \cP_d | \max_{P \in \cP_d} \p ( \mathsf{weight}( P) \le t ) \nonumber \\
& \le \sum_{d \ge \beta t} | \cP_d | \sqrt{d}  \left( \frac{2.8 \lambda_{max} }{\beta} \right)^d \nonumber \\
\label{eq:weight_summation}
& \le \sum_{d \ge \beta t} | \cP_d | \sqrt{d} \left( \frac{2.8}{3 \Delta } \right)^d,
\end{align}
where the first inequality is due to a union bound, and
the third inequality  holds for $t$ sufficiently large and whenever $\beta > \lambda_{max}$, in light of Lemma \ref{lemma:weight_lower_bound}, and the final expression follows from substituting $\beta = 3 \Delta \lambda_{max}$. To simplify the final summation, we will bound the size of $\cP_d$. Since the maximum degree in $G$ is $\Delta$, simple counting arguments show that there are at most $\Delta^d$ paths of length $d$. For sufficiently large $t$, the summation in \eqref{eq:weight_summation} can be bounded by 
\begin{align*}
\sum_{d \ge \beta t} \sqrt{d} \left( \frac{2.8}{3 \Delta} \right)^d & \le \sum_{d \ge \beta t} \left( \frac{ 2.9  }{3 \Delta} \right)^d  = 30 \left(\frac{ 2.9 }{3} \right)^{\beta t}.
\end{align*}
In the display above, the inequality holds for $t$ sufficiently large, and the final expression is follows from the formula for the sum of a geometric series.

Putting everything together, we have for $k$ sufficiently large that 
\begin{align*}
\p ( \cE_{2,k}^c) & \le \p \left( \exists t \ge k : \cascade(t) \setminus \cN_{v^*}(\beta t) \neq \emptyset \right) \\
& \le \sum_{t \ge k} \p ( \cascade(t) \setminus \cN_{v^*}(\beta t) \neq \emptyset ) \\
& \le \sum_{t \ge k} 30 \left( \frac{2.9}{3} \right)^{\beta t} = 900 \left( \frac{2.9}{3} \right)^{\beta k}.
\end{align*}
\end{proof}

\section{Properties of the score functions}

Throughout this appendix, we will condition on the cascade evolution, denoted by $\boldsymbol{\cascade} : = \{ \cascade(s) \}_{s \ge 0}$. We will also assume that $\boldsymbol{\cascade}$ is a realization of the cascade where the event $\cE_k$ holds (see Proposition \ref{prop:Ek}), for an appropriately chosen value of $k$.
We will also define the measure $\p_v ( \cdot) = \p ( \cdot \vert v^* = v)$.

We begin by establishing a useful representation of the difference between score functions.
Let $\{A_i \}_{i \ge 1}, \{ B_i \}_{i \ge 1}$ be independent collections of i.i.d. $Q^+$-distributed random variables. 
For any $v \in V(G)$, let $g_v^\alpha(t) : = \sum_{s = 0}^t | \cN_v(\alpha s) \cap \cascade(s) |$.
For $u,v \in V$ and any $t \ge 0$, define the processes 
\begin{align}
\label{eq:M_definition}
M_{vu}^\alpha(t) & := \sum_{i = 1}^{f_v^\alpha(t) + f_u^\alpha(t) } A_i \\
\label{eq:D_definition}
D_{vu}^\alpha(t) & : = \sum_{i = 1}^{ f_v^\alpha(t) - g_v^\alpha(t) } (A_i + B_i). 
\end{align}

\begin{lemma}
\label{lemma:Zv_Zu_representation}
For $t$ satisfying $0 \le t \le \dist (v,u) / ( \alpha + \beta)$, 
\[
Z_v(t) - Z_u(t) \stackrel{d}{=} M_{vu}^\alpha(t) + D_{vu}^\alpha(t),
\]
with respect to the measure $\p_v ( \cdot \vert \boldsymbol{\cascade} )$.
\end{lemma}

\begin{proof}
We first consider the distributional representation of $Z_v(t)$ with respect to $\p_v ( \cdot \vert \boldsymbol{\cascade})$. From the definition of the signal model, we have for any $t \ge 0$ that 
\begin{align*}
Z_v(t) & = \sum_{s = 0}^t \left( \sum_{w \in  \cN_v(\alpha s) \cap \cascade(s) } Y_w(s) + \sum_{w \in \cN_v(\alpha s) \setminus \cascade(s) } Y_w(s) \right ) \\
& \stackrel{d}{=} \sum_{i = 1}^{g_{v}^\alpha(t) } A_i - \sum_{i = 1}^{f_v^\alpha( t) - g_{v}^\alpha(t) } B_i,
\end{align*}
where $\{ A_i \}_{i \ge 1}$ and $\{ B_i \}_{i \ge 1}$ are collections of i.i.d. samples from $Q^+$. On the other hand, notice that if the event $\cE_k$ holds, then $\cascade(t) \subseteq \cN_v(\beta t)$ for $t \ge k$, hence $\cN_u( \alpha t) \cap \cascade(t) = \emptyset$ for $u \in V_n \setminus \cN_v( (\alpha + \beta) t)$. 
As a consequence, 
\[
Z_u(t) \stackrel{d}{=} - \sum_{i = 1}^{f_u^\alpha(t) } A_i .
\]
Subtracting the two expressions, we see that for $t \ge k$, 
\[
Z_v(t) - Z_u(t) \stackrel{d}{=} \sum_{i = 1}^{g_v^\alpha(t) + f_u^\alpha(t) } A_i - \sum_{i = 1}^{f_v^\alpha(t) - g_v^\alpha(t) } B_i.
\]
Letting $\{ A_i' \}_{i \ge 1}$ be another sequence of i.i.d. $Q^+$-distributed random variables, we can add and subtract terms to obtain 
\begin{multline*}
Z_v(t) - Z_u(t) \stackrel{d}{=} \left( \sum_{i = 1}^{f_v^\alpha(t) - g_v^\alpha(t)} A_i + \sum_{i = 1}^{g_v^\alpha(t) + f_u^\alpha(t)} A_i \right) \\
- \sum_{i = 1}^{f_v^\alpha(t) - g_v^\alpha(t)} ( A_i' + B_i).
\end{multline*}
It is readily seen that this expression is equal to the one in the statement in Lemma \ref{lemma:Zv_Zu_representation}.
\end{proof}

We proceed by establishing some useful properties of $M_{vu}^\alpha(t)$ and $D_{vu}^\alpha(t)$. 

\begin{lemma}
\label{lemma:D_bound}
On the event $\cE_k$, it holds that $\sup_{t \ge 0} |D_{vu}^\alpha(t)| \le 2f_v^\alpha(k)$ almost surely.
\end{lemma}

\begin{proof}
Conditionally on $\boldsymbol{\cascade}$, it follows from the definition of $D_{vu}^\alpha(t)$ that $|D_{vu}^\alpha(t) | \le 2(f_v^\alpha(t) - g_v^\alpha(t))$. On the event $\cE_k$, we have that $\cN_v(\alpha t) = \cN_v(\alpha t) \cap \cascade(t)$ for all $t \ge k$, so 
\[
f_v^\alpha(t) - g_v^\alpha(t) = f_v^\alpha(k) - g_v^\alpha(k) \le f_v^\alpha(k).
\]
The desired claim follows.
\end{proof}

\begin{lemma}
\label{lemma:M_concentration}
For $k$ sufficiently large, it holds for $t \ge 12 k / (p (1 - 2 \epsilon))$ that 
\begin{multline*}
\p_v \left( \left. Z_v(t) - Z_u(t) \le \frac{p (1 - 2 \epsilon) }{3} f_v^\alpha(t) \right \vert \cE_k \right)  \\
\le 2 \exp \left(  - \frac{p (1 - 2 \epsilon)^2}{10} f_v^\alpha(t) \right) . 
\end{multline*}
\end{lemma}

\begin{proof}
Condition on $\boldsymbol{\cascade}$ and suppose that the event $\cE_k$ holds.
As a shorthand, let $N : = f_v^\alpha(t) + f_u^\alpha(t)$.
Under $\p_v$, $M_{vu}^\alpha(t)$ is a sum of $N$ i.i.d. $Q^+$-distributed random variables. For $A \sim Q^+$, we have that $\E [ A ] = p(1 - 2 \epsilon)$, which in turn implies $\E_v [ M_{vu}^\alpha(t) \vert \boldsymbol{\cascade} ] = p ( 1 - 2\epsilon) N$.
It also holds that $\mathrm{Var}(A) \le \E [ A^2] = p$. Bernstein's inequality therefore implies that
\begin{align}
\label{eq:M_concentration}
\p_v \left( \left. M_{vu}^\alpha(t) \le \frac{p ( 1 - 2 \epsilon)}{2} N \right \vert \boldsymbol{\cascade} \right) & \le \exp \left( - \frac{ p (1 - 2 \epsilon)^2}{10} N \right).
\end{align}
We now use \eqref{eq:M_concentration} to derive an inequality for $Z_v(t) - Z_u(t)$. We can write
\begin{align*}
\p_v & \left( \left. Z_v(t) - Z_u(t) \le \frac{p (1 - 2 \epsilon)}{3} f_v^\alpha(t) \right \vert \boldsymbol{\cascade} \right) \\
& \stackrel{(a)}{\le} \p_v \left( \left. M_{vu}^\alpha(t) \le \frac{p(1 - 2 \epsilon)}{3} f_v^\alpha(t) + 2 f_v^\alpha(k) \right \vert \boldsymbol{\cascade} \right) \\
& \stackrel{(b)}{\le} \p_v \left(\left.  M_{vu}^\alpha(t) \le \frac{p(1 - 2 \epsilon)}{2} N \right \vert \boldsymbol{\cascade} \right) \\
& \stackrel{(c)}{\le} \exp \left ( - \frac{p(1 - 2 \epsilon)^2}{10} f_v^\alpha(t) \right ).
\end{align*}
Above, $(a)$ is due to the decomposition $Z_v(t) - Z_u(t) = M_{vu}^\alpha(t) + D_{vu}^\alpha(t)$ and since $|D_{vu}^\alpha(t)| \le 2 f_v^\alpha(k)$ on the event $\cE_k$ by Lemma \ref{lemma:D_bound}; 
$(b)$ follows since $N \ge f_v^\alpha(t)$ assumes $t$ is sufficiently large so that $f_v^\alpha(k) \le p(1 - 2 \epsilon) f_v^\alpha(t) / 6$, which holds when $t \ge 12k / ( p(1 - 2 \epsilon))$ in light of Lemma \ref{lemma:f_comparison}; 
and $(c)$ uses \eqref{eq:M_concentration} and $N \ge f_v^\alpha(t)$.

To relate the probability bound to the probability of interest, we can write, for any event $\cA$, 
\[
\p_v ( \cA \vert \cE_k) = \frac{\p_v ( \cA \cap \cE_k) }{\p ( \cE_k)} = \frac{ \E_v [ \p_v ( \cA \vert \boldsymbol{\cascade} ) \mathbf{1}( \cE_k)] }{ \p(\cE_k)}.
\]
The desired result follows from letting $\cA : = \{ Z_v(t) - Z_u(t) \le p ( 1 - 2 \epsilon ) f_v^\alpha(t) / 3 \}$ and from letting $k$ be sufficiently large so that $\p( \cE_k) \ge 1/2$.
\end{proof}

\begin{lemma}
\label{lemma:M_martingale}
Fix $u,v \in V(G)$, and let $T$ be a stopping time satisfying $T \le \dist(u,v) / (\alpha + \beta)$ almost surely. Then for any sufficiently large integer $k$ and any $x \ge 0$, it holds that
\[
\p_v \left( \left. Z_v(T) - Z_u(T) \le - x \right \vert \cE_k \right) \le 2\left( \frac{\epsilon}{1 - \epsilon} \right)^{x - 2f_v^\alpha(k)}.
\]
\end{lemma}

\begin{proof}
Suppose that $\cE_k$ holds. Conditionally on $\boldsymbol{\cascade}$, we claim that the process $(\epsilon / (1 - \epsilon) )^{M_{vu}^\alpha(t)}$
is a martingale with respect to $\p_v ( \cdot \vert \boldsymbol{\cascade})$. To see why, notice that for $A \sim Q^+$, 
we have that 
\[
\E \left[ \left( \frac{\epsilon}{1 - \epsilon} \right)^A \right] = 1.
\]
Since $M_{vu}^\alpha(t)$ is an i.i.d. sum of $f_v^\alpha(t) + f_u^\alpha(t)$ random variables distributed according to $Q^+$, the claim follows.

We now turn to the proof of the main result. Let $T$ be any stopping time that is almost surely bounded. Then, conditioned on $\boldsymbol{\cascade}$, we have that
\begin{align*}
\p_v & \left( Z_v(T) - Z_u(T) \le - x \vert \boldsymbol{\cascade} \right) \\
& \stackrel{(a)}{\le} \p_v ( M_{vu}^\alpha(T) \le - x + 2 f_v^\alpha(k) \vert \boldsymbol{\cascade} ) \\
& = \p_v \left( \left. \left( \frac{\epsilon}{1 - \epsilon} \right)^{M_{vu}^\alpha(T)} \ge \left( \frac{1 - \epsilon}{\epsilon} \right)^{x - 2f_v^\alpha(k) } \right \vert \boldsymbol{\cascade} \right) \\
& \stackrel{(b)}{\le} \left( \frac{\epsilon}{1 - \epsilon} \right)^{x- 2f_v^\alpha(k)} \E_v \left[ \left. \left( \frac{\epsilon}{1 - \epsilon} \right)^{M_{vu}^\alpha(T)} \right \vert \boldsymbol{\cascade} \right ] \\
& \stackrel{(c)}{=} \left( \frac{\epsilon}{1  - \epsilon} \right)^{x - 2f_v^\alpha(k) }.
\end{align*}
Above, $(a)$ follows from the representation of $Z_v(T) - Z_u(T)$ in Lemma \ref{lemma:Zv_Zu_representation} (which holds since $T \le \dist(u,v) / (\alpha + \beta)$ almost surely) as well as the bound for $D_{vu}^\alpha(T)$ established in Lemma \ref{lemma:D_bound};
$(b)$ is due to Markov's inequality; 
and $(c)$ is a consequence of the Optional Stopping Theorem \cite[Section 10.10]{williams_1991}, since $T$ is almost surely bounded. 

Finally, to replace the conditioning on the cascade with conditioning on $\cE_k$, we may follow the same reasoning as the proof of Lemma \ref{lemma:M_concentration}.
\end{proof}

\section{Proofs of Theorem \ref{thm:main_performance_analysis} and Corollary \ref{cor:cascade_estimator_size}}
\label{sec:proof_of_thm}

We prove a series of lemmas which establish properties of $T^{\alpha \beta}, \dist(v^*, \widehat{v}(T^{\alpha \beta}) )$ and $\widehat{\cascade}(T^{\alpha \beta})$. The proof of Theorem \ref{thm:main_performance_analysis}, which follows readily from these results, can be found at the end of this section.

\begin{lemma}
\label{lemma:T_tail_bound}
Let $k$ be sufficiently large and let $v \in U$. Then for $|U|$ sufficiently large,
\[
\p\left( \left. T^{\alpha \beta} \ge F_{v^*}^\alpha \left( \frac{15 \log |U| }{p (1 - 2 \epsilon)^2} \right)  \right \vert \cE_k \right) \le 2 |U|^{-1/2}.
\]
\end{lemma}

\begin{proof}
Define the quantity 
\begin{equation}
\label{eq:tn}
t_v : = F_v^\alpha \left( \frac{15 \log |U| }{p (1 - 2 \epsilon)^2} \right),
\end{equation}
and notice that
\begin{equation}
\label{eq:fv_tau_inequality}
\frac{p(1 - 2 \epsilon)}{3} f_v^\alpha(t_v) \ge \frac{ 5 \log |U| }{1 - 2 \epsilon} \ge \frac{ 5 \log |U| }{\log \left( \frac{1 - \epsilon}{\epsilon} \right)} \ge \tau.
\end{equation}
Moreover, we can bound the probability of interest as
\begin{align}
\p_v & ( T^{\alpha \beta} \ge t_v \vert \cE_k )  \le \p_v ( T^{\alpha \beta}(v) \ge t_v  \vert \cE_k ) \nonumber \\
& \stackrel{(a)}{\le} \sum_{u \in U \setminus \cN_v( ( \alpha + \beta) t_v )} \hspace{-0.7cm} \p_v ( Z_v(t_v ) - Z_u(t_v ) < \tau \vert \cE_k ) \nonumber  \\
& \stackrel{(b)}{\le} 2 \exp \left ( \log |U| - \frac{p (1 - 2 \epsilon)^2}{10} f_v^\alpha(t_v) \right) \nonumber \\
& \stackrel{(c)}{\le} 2 \exp \left ( - \frac{1}{2} \log |U| \right ) . \nonumber 
\end{align}
Above,
$(a)$ is due to a union bound,
$(b)$ is a consequence of Lemma \ref{lemma:M_concentration} since \eqref{eq:fv_tau_inequality} holds, and
$(c)$ uses the definition of $t_v$.
\end{proof}

\begin{lemma}
\label{lemma:estimation_tail_bound}
Let $k$ be sufficiently large. Then for any $v \in U$, it holds that
\[
\p_v \left( \left. \dist(v, \widehat{v}^{\alpha \beta} ) > (\alpha + \beta) T^{\alpha \beta} \right \vert \cE_k \right) \le c_{k,\epsilon} |U|^{-1},
\]
where $c_{k,\epsilon}$ is a constant depending only on $k$ and $\epsilon$.
\end{lemma}

\begin{proof}
We start by writing
\begin{multline}
\label{eq:expected_error_pt2}
\p_v \left( \left. \dist(v, \widehat{v}^{\alpha \beta} ) \ge (\alpha + \beta) T^{\alpha \beta} \right \vert \cE_k \right) \\
= \sum_{u \in U} \p_v \left( \left. \widehat{v}^{\alpha \beta} = u, T^{\alpha \beta} < \frac{ \dist(v,u)}{\alpha + \beta} \right \vert \cE_k \right).
\end{multline}
To analyze the probabilities in the summation, let us define the stopping time $T_{vu} : = \max \{ T^{\alpha \beta}, \dist(v,u) / ( \alpha + \beta) \}$. Then we can bound
\begin{align}
 \p_v & \left( \left. \widehat{v}^{\alpha \beta} = u, T^{\alpha \beta} < \frac{ \dist(v,u) }{\alpha + \beta} \right \vert \cE_k \right) \nonumber \\
& \le \p_v ( Z_v(T_{vu}) - Z_u(T_{vu}) \le - \tau \vert \cE_k ) \nonumber \\
\label{eq:expected_error_probability_bound}
& \le \left( \frac{\epsilon}{1 - \epsilon} \right)^{\tau - f_v^\alpha(k) } = : \left( \frac{1-\epsilon}{\epsilon} \right)^{f_v^\alpha(k)} |U|^{-2}.
\end{align}
Above, the first inequality follows since $T_{vu} = T^{\alpha \beta}$ when $T^{\alpha \beta} < \dist(u,v) / (\alpha + \beta)$, and the second inequality is a consequence of Lemma \ref{lemma:M_martingale}. In the final expression, we substitute $\tau = 2 \log |U| / \log ( (1 - \epsilon) / \epsilon)$.
Finally, substituting the bound in \eqref{eq:expected_error_probability_bound} into \eqref{eq:expected_error_pt2} shows that the probability of interest is at most $c_{k,\epsilon} |U|^{-1}$, where $c_{k,\epsilon} : = ((1 - \epsilon)/\epsilon)^{f_v^\alpha(k)}$.
\end{proof}

\begin{lemma}
\label{lemma:cascade_estimator}
Recall that 
\[
\widehat{\cascade}(T^{\alpha \beta}) := \cN_{\widehat{v}(T^{\alpha \beta})} ( (\alpha + 2 \beta) T^{\alpha \beta} ).
\]
For $k$ sufficiently large, and all $|U|$ sufficiently larger than $k$, it holds that 
\[
\p \left( \left.  \cascade(T^{\alpha \beta}) \not \subseteq \widehat{\cascade}(T^{\alpha \beta} ) \right \vert \cE_k \right) \le c_{k,\epsilon} |U|^{-1},
\]
where $c_{k,\epsilon}$ is the constant from Lemma \ref{lemma:estimation_tail_bound}.
\end{lemma}

\begin{proof}
Observe that, for $k$ fixed, we have that $T^{\alpha \beta} \ge k$ when $|U|$ is sufficiently large. Indeed, for any $u,v \in U$, we have that $|Z_v(t) - Z_u(t) | \le f_v^\alpha(t) + f_u^\alpha(t)$ as all the observed signals are at most 1 in absolute value. Hence, for all $0 \le t \le k$, $|Z_v(t) - Z_u(t) | \le f_v^\alpha(k) + f_u^\alpha(k)$, which is smaller than the threshold $2 \log |U| / \log ( (1 - \epsilon ) / \epsilon )$ for $|U|$ sufficiently large. It follows immediately that $T^{\alpha \beta}$ must be larger than $k$.

Next, for $k$ large enough, we have that 
\begin{align*}
\p_v&  \left( \left. |\cascade(T^{\alpha \beta}) \setminus \widehat{\cascade}(T^{\alpha \beta}) | \ge 1 \right \vert \cE_k \right) \\
& \stackrel{(a)}{\le} \p_v \left( \left.  \exists u \in \cN_v( \beta T^{\alpha \beta} ) \setminus \cN_{\widehat{v}^{\alpha \beta}} ( (\alpha + 2 \beta) T^{\alpha \beta} ) \right \vert \cE_k \right) \\
& \stackrel{(b)}{\le} \p_v \left( \left. \dist ( v, \widehat{v}^{\alpha \beta} ) \ge ( \alpha + \beta) T^{\alpha \beta} \right \vert \cE_k \right) \\
& \stackrel{(c)}{\le} c_{k,\epsilon} |U|^{-1}.
\end{align*}
Above, $(a)$ follows since $\cascade(t) \subseteq \cN_v(\beta t)$ for all $t \ge k$ on $\cE_k$ when $v$ is the source; 
$(b)$ is the due to the triangle inequality;
and $(c)$ is a consequence of Lemma \ref{lemma:estimation_tail_bound}.
\end{proof}

\begin{proof}[Proof of Theorem \ref{thm:main_performance_analysis}]
Notice that for $k$ fixed but sufficiently large, the probability bounds in Lemmas \ref{lemma:T_tail_bound}, \ref{lemma:estimation_tail_bound}, and \ref{lemma:cascade_estimator} tend to zero as $|U| \to 0$. We conclude by noting that $\p( \cE_k) \to 1$ by Proposition \ref{prop:Ek}, so we may take $k$ to increase at an arbitrarily slow rate with $|U|$ to remove the conditioning on $\cE_k$ in Lemmas \ref{lemma:T_tail_bound}, \ref{lemma:estimation_tail_bound} and \ref{lemma:cascade_estimator}. The theorem now follows from a union bound.
\end{proof}

\begin{proof}[Proof of Corollary \ref{cor:cascade_estimator_size}]
By the form of the cascade estimator (see \eqref{eq:cascade_estimator}) and by Item \ref{item:thm_dist} of Theorem \ref{thm:main_performance_analysis}, we have that $\widehat{\cascade}(T^{\alpha \beta})$ is contained within $\cN_{v^*} ( (\alpha + \beta) T^{\alpha \beta} )$ with probability tending to 1 as $|U| \to \infty$. Moreover, it holds that
\begin{align*}
\left | \widehat{\cascade}(T^{\alpha \beta} )\right| &  \stackrel{(a)}{\le} \left|  \cN_{v^*} \left( (\alpha + \beta) F_{v^*}^\alpha \left( \frac{ 15 \log |U| }{p (1 - 2 \epsilon)^2 } \right) \right) \right| \\
& \stackrel{(b)}{\le} \left( q \left| \cN_{v^*} \left( \alpha F_v^\alpha \left( \frac{15 \log |U| }{p (1 - 2 \epsilon)^2 } \right) \right) \right| \right)^{r (\alpha + \beta)/\alpha} \\
& \stackrel{(c)}{\le} \left( q f_v^\alpha \left( F_v^\alpha \left( \frac{15 \log |U| }{p (1 - 2 \epsilon)^2 } \right) \right) \right)^{r ( \alpha + \beta) / \alpha} \\
& \stackrel{(d)}{=} \left( \frac{15 q \log |U| }{p (1 - 2 \epsilon)^2} \right)^{r ( \alpha + \beta) / \alpha}.
\end{align*}
Above, $(a)$ follows from Item \ref{item:thm_T} of Theorem \ref{thm:main_performance_analysis};
$(b)$ is due to Lemma \ref{lemma:neighborhood_power_bound};
$(c)$ follows since $f_v^\alpha(t) \ge \cN_v(\alpha t)$ for any $t \ge 0$ and $v \in V(G)$;
and $(d)$ uses that $F_v^\alpha$ is the inverse function of $f_v^\alpha$.
The desired claim follows from noting that the final bound is at most $\log^c |U|$ when $c > r (\alpha + \beta ) / \alpha$.
\end{proof}

\section{Basic graph-theoretic results}

In this appendix, we prove a few elementary but useful graph-theoretic results. 

\begin{lemma}
\label{lemma:neighborhood_bound}
Suppose that $G$ is connected and has maximum degree $\Delta$. Then for all $v \in V(G)$ and all $t \ge 0$, it holds that $| \cN_v(t) | \le 2 \Delta^{t}$.
\end{lemma}

\begin{proof}
Let $\partial \cN_v(t)$ denote the set of vertices of distance \emph{exactly} $t$ from $v$. Then we have the inequality $| \partial \cN_v(t) | \le \Delta | \partial \cN_v(t-1)|$ for all $t \ge 1$, which implies that $| \partial \cN_v(t) | \le \Delta^t$. We then have that
\begin{align*}
| \cN_v(t) | & = \sum_{s = 0}^t | \partial \cN_v(s) |  \le \sum_{s = 0}^t \Delta^s \\
& = \frac{\Delta^{t + 1} - 1}{\Delta - 1} \le \left( \frac{\Delta}{\Delta - 1} \right) \Delta^t.
\end{align*}
Since $G$ is connected, $\Delta \ge 2$. As a result, $\Delta / (\Delta - 1) \le 2$, which proves the desired statement.
\end{proof}

\begin{lemma}
\label{lemma:f_comparison}
Let $\alpha > 0$ and $v \in V(G)$.
For any $1 \le t_1 \le t_2$, it holds that
\[
f_v^\alpha(t_1) \le \frac{2t_1 }{t_2} f_v^\alpha(t_2).
\]
\end{lemma}

\begin{proof}
Since $f_v^\alpha(t)$ is a sum of increasing terms, we have that $f_v^\alpha (t_1) / (t_1 + 1) \le f_v^\alpha(t_2) / (t_2 + 1)$. It follows that 
\[
f_v^\alpha(t_1) \le \frac{t_1 + 1}{t_2 + 1} f_v^\alpha(t_2).
\]
The desired result follows since $t_1 + 1 \le 2t_1$ and $t_2 + 1 \ge t_2$.
\end{proof}

\begin{lemma}
\label{lemma:f_difference}
For any $0 \le t_1 \le t_2$, it holds that $f_v^\alpha(t_2) - f_v^\alpha(t_1) \ge t_2 - t_1$.
\end{lemma}

\begin{proof}
Since $|\cN_v(\alpha s) | \ge 1$ for all $s \ge 0$, it holds that
\[
f_v^\alpha(t_2) - f_v^\alpha(t_1) = \sum_{s = t_1 + 1}^{t_2} | \cN_v(\alpha s) | \ge t_2 - t_1.
\]
\end{proof}

\begin{lemma}
\label{lemma:Fv_asymptotics}
As $z \to \infty$, $F_v^\alpha(z) \sim F_v(\alpha z) / \alpha$.
\end{lemma}

\begin{proof}
Let $\tilde{f}_v^\alpha(t)$ and $\tilde{f}_v(t)$ denote the continuous-time versions of $f_v^\alpha(t)$ and $f_v(t)$ (which are well-defined for $t \in \Z_{\ge 0}$) formed by linear interpolation. 
To relate $\tilde{f}_v^\alpha$ and $\tilde{f}_v$, we can write
\begin{equation}
\label{eq:fv_alpha_relation}
\tilde{f}_v^\alpha(t) = \int_0^t | \cN_v(\alpha s) | ds = \frac{1}{\alpha} \int_0^{\alpha t} | \cN_v(\theta) | d\theta = \frac{1}{\alpha} \tilde{f}_v(\alpha t).
\end{equation}
Next, define $\tilde{F}_v^\alpha$ and $\tilde{F}_v$ to be the inverse functions of $\tilde{f}_v^\alpha$ and $\tilde{f}_v$, respectively. It holds for any $t \ge 0$ that 
\[
\frac{1}{\alpha} \tilde{F}_v( \alpha \tilde{f}_v^\alpha(t) ) = \frac{1}{\alpha} \tilde{F}_v ( \tilde{f}_v(\alpha t) ) = t = \tilde{F}_v^\alpha( \tilde{f}_v^\alpha(t) ),
\]
where the first equality is due to the relation \eqref{eq:fv_alpha_relation}. As $\tilde{f}_v^\alpha$ is a continuous function, it follows that for any $z \ge 0$, $\tilde{F}_v^\alpha(z) = \tilde{F}_v( \alpha z) / \alpha$.

Finally, notice that $\tilde{F}_v^\alpha$ and $\tilde{F}_v$ are linearly interpolated versions of $F_v^\alpha$ and $F_v$. Moreover, $\tilde{F}_v^\alpha$ and $\tilde{F}_v$ grow at most linearly in light of Lemma \ref{lemma:f_difference}. Hence $F_v^\alpha(z) \sim \tilde{F}_v^\alpha(z)$ and $F_v(z) \sim \tilde{F}_v(z)$ as $z \to \infty$, and the desired result follows. 
\end{proof}

\begin{lemma}
\label{lemma:neighborhood_power_bound}
Assume that the condition \eqref{eq:structural_constraint} holds. Then for non-negative integers $t$ and $k$, it holds for any $v \in V(G)$ that $| \cN_v(kt) | \le ( q | \cN_v(t) |^r)^k$.
\end{lemma}

\begin{proof}
Let $k \ge 0$ be an integer. Noting that $\cN_v( (k + 1) t) \subseteq \bigcup_{u \in \cN_v(kt)} \cN_u(t)$, it holds that
\begin{multline*}
| \cN_v((k + 1)t) | \le \sum_{u \in \cN_v(kt)} | \cN_u(t) | \le  q | \cN_v(t) |^r | \cN_v(kt) |,
\end{multline*}
where the first inequality is due to a union bound, and the second is due to \eqref{eq:structural_constraint}. Solving this recursion with the initial condition $| \cN_v(0) | = 1$ proves the lemma.
\end{proof}

\end{document}